\documentclass{article}
\usepackage{graphicx,color, amsthm, amsmath, amssymb, amsfonts, float} 
\usepackage[a4paper, total={6in, 8in}]{geometry}
\usepackage{algpseudocode, algorithm}
\usepackage{pgfplots, caption}
\usepackage{hyperref, tikz, array}
\pgfplotsset{compat=1.18} 
\usepackage{tikz-cd}

\newcommand{\cC}{\mathcal C}

\newcommand{\tr}{\text{Trop}}
\title{On the Genus of One Degree of Freedom Planar Linkages \\via Tropical Geometry}
\newtheorem{theorem}{Theorem}
\newtheorem{corollary}{Corollary}
\newtheorem{lemma}{Lemma}
\newtheorem{proposition}{Proposition}

\theoremstyle{definition}
\newtheorem{definition}{Definition}
\newtheorem{example}{Example}
\newtheorem{remark}{Remark}

\author{J. Schicho, A. K. Tewari, A. Warren}
\date{\today}

\begin{document}

\maketitle

\begin{abstract}
This paper focuses on studying the configuration spaces of graphs realised in $\mathbb C^2$, such that the configuration space is, after normalisation, one dimensional. If this is the case, then the configuration space is, generically, a smooth complex curve, and can be seen as a Riemann surface. The property of interest in this paper is the genus of this curve. Using tropical geometry, we give an algorithm to compute this genus. We provide an implementation in Python and give various examples.   
\end{abstract}

\section*{Introduction}

Let $G=(V,E)$ be an undirected simple graph without self-loops. Intuitively, a planar realisation of $G$ consists of a point in the real plane for every vertex and a straight line segment for every edge (edges may intersect). Let us prescribe a real number for every edge. Then a realisation such that the squared length of every edge is equal to the prescribed number is called a {\em configuration} of the graph (together with the assignment of numbers to edges). The set of configurations is an algebraic variety given by $|E|$ equations in $2|V|$ variables. If we have any realisation, then any translation or rotation of a configuration is again a configuration for the same edge numbers. We can normalise: pick an edge and fix its endpoints beforehand. The normed realisations form an algebraic variety with $|E|-1$ equations in $2|V|-4$ variables.

In this paper, we look at the complexification of the variety of normed realisations (given by the same equations, but variables range over the complex numbers). If the assigned edge numbers are generic, then the set of normalised configurations is a nonsingular algebraic variety of dimension $2|V|-|E|-3$, if it is not empty. The case of minimally rigid graphs is classical (see \cite{borcea2004laman,jackson2019laman}): in this case, $2|V|-|E|-3=0$, and for generic edge numbers we have a finite set of configurations. This finite number is independent of the choice of edge numbers as long as they are generic, and it is therefore an invariant of the graph called the {\em Laman number}. A combinatorial algorithm for computing the Laman number of a given minimally rigid graph is given in \cite{capco2018number}.

A one degree of freedom linkage/graph (1-dof, also called a calligraph), is obtained by removing one edge from a minimally rigid graph. In this case, the configuration set is a nonsingular algebraic variety of dimension~1. It may be reducible; the number of irreducible components is equal to the product of all maximal subgraphs that are minimally rigid, by \cite{lubbes2024irreducible}. Our result here is an algorithm that computes this genus. Notice that the genus does not depend on the choice of edge numbers as long as they are chosen generically, because each generic fiber is defined by essentially the same equations over the field generated by the edge numbers. As an application, we give an answer to ``Open Problem 3'' in \cite{lubbes2024open}, which asks for the genus of the plane curve traced by a single vertex while an edge is kept fixed. A classically known case is the 4-cycle, also known as four-bar linkage. Here, the configuration set is an irreducible curve of genus one (see \cite{izmestiev2020elliptic}).

Tropicalization has been used in \cite{capco2018number,dewar2024tropical,bernstein2019tropical} to prove statements about Laman numbers of minimally rigid graphs.
By some linear change of variables, one can reduce the system of algebraic equations that defines the configuration set of a graph to a system of linear equations that depend on the graph together with equations $s_et_e=\lambda_e$, with $e$ ranging over all edges and $s_e,t_e$ are indeterminates. The tropicalization process transforms these quadratic equations into linear equations of the form $x_e+y_e=w_e$, where $x_e$ and $y_e$ are tropical variables
and $w_e$ is a tropical (generic) parameter. The tropicalized system of equations consists of linear equations and inequalities, and its solution set is a tropical variety, the tropicalized configuration set. If $G$ is minimally rigid, then the tropical variety is a finite set of points, and its cardinality is the Laman number.
If $G$ is a 1-dof graph, then we have a disjoint union of tropical curves. We show that the genus of an irreducible component of the fiber is equal to the genus of the any of these tropical curves. Note that in general, the genus of an algebraic curve may be different from the genus of its tropicalization; however, if the tropicalization is tropically smooth, or equivalently if the algebraic curve is a Mumford curve, then the algebraic genus and the tropical genus coincide \cite{jell2020constructing}. 

Our algorithm computes the tropical curve explicitly; once this is achieved, computing the genus amounts to counting edges and vertices. We compute a point on the tropicalization of the generic configuration space by a probabilistic algorithm. Then we compute the directions of the edges incident to this point and follow them to the next vertex. We repeat this until no new edges and vertices are found, iteratively traversing the whole tropical curve. 

For each new vertex we find, we test a condition that implies that the tropical curve is smooth at that vertex, namely the transversality of the intersection of the Bergman fan of the graph and a mirror copy of the Bergman fan. If all tests are affirmative, then we compute the genus. In the unlikely case that we receive a negative test result, we give an error message. In this case, we can start over again by computing probabilistically a new starting point.

We have implemented the algorithm in Python; the code can be found at the GitHub repository \url{AudieWarren/1dof-Graph-Genus-Code}.
In order to use the code, download the folder `Modularversion'. In the file `graphgeneration' you can input the edges of your graph, with many examples already being in the file. Then simply run the `main' file; the programme will output the genus of the curve, together with a visualisation of the tropical curve. Note that you may have to install Networkx via pip for the code to run.

\section{The Precise Statement of the Problem, \\ and a Theorem on Tropical Curves}

\subsection{The Precise Statement of the Problem}

For any graph $G=(V,E)$ on $n \geq 2$ vertices with at least one edge, we label the vertices so that $\{1,2\} \in E$. In order to analyse the realisations of this graph, we fix the vertex $1$ at the origin, and the vertex $2$ to the point $(1,0)$. This removes the trivial motions of the graph via elements of $SE(2)$, and allows us to define the two-dimensional (normalised) edge map of $G$ in the following way.
\begin{gather*}
    f_G:\mathbb C^{2|V|-4} \rightarrow \mathbb C^{|E|-1} \\
    (x_i,y_i)_{i \in V \setminus \{1,2\}} \mapsto ((x_i-x_j)^2 + (y_i-y_j)^2)_{i,j \in E\setminus\{1,2\}}.
\end{gather*}
This map takes input positions in $\mathbb C^2$ for the vertices, and calculates the squared distances along edges. In this paper, we define a \textit{generic} element of $\mathbb C^{d}$ (for any $d)$ to be any point whose coordinate entries are algebraically independent over $\mathbb Q$. Any two generic elements are algebraically indistinguishable; more precisely, there are is a field isomorphism that preserves any variety/map defined over $\mathbb Q$ taking the first generic point to the second one. In this way, we study the \textit{generic fiber} of the edge map $f_G$ - this means the fiber of any generic element of $\mathbb C^{|E|-1}$. This generic fiber is also referred to as the generic (normalised, complex) configuration space of $G$. If the map $f_G$ is generically surjective, then the generic fiber is a smooth manifold of dimension $2|V|-|E|-3$. In this paper, we assume $2|V|-|E|=4$, in addition to the surjectivity of $f_G$. This is equivalent to the statement that $G$ can be obtained from a minimally rigid graph by removing a single edge. The graph can then be considered as a one degree of freedom linkage.

The number of irreducible components of the generic fiber is well-known: it is the product of the Laman numbers of all maximal subgraphs that are minimally rigid. We provide an algorithm that computes the genus of any irreducible component for a given graph that fulfills the above assumption. All irreducible components have the same genus, by \cite[Proposition~38]{lubbes2024irreducible}.

\subsection{A Theorem on Tropical Curves}

Tropical geometry is the piecewise linear shadow of algebraic geometry where with a modified notion of vanishing, the vanishing locus of tropical polynomials are studied with respect to the semiring $\mathbb{T} = (\mathbb{R} \cup \infty, \text{min}, +)$. We refer the reader to the standard text \cite{sturmfelsmaclagan} for basics concerning tropical geometry. A \emph{tropical curve} is a connected tropical variety of dimension one. By the structure theorem of tropical varieties \cite[Theorem 3.3.5]{sturmfelsmaclagan}, a tropical curve also enjoys the structure of a balanced rational polyhedral complex. For instance, planar tropical curves are dual to regular subdivisions of the corresponding Newton polygon. 

Given an irreducible algebraic curve, we can `tropicalize' the curve, which yields a tropical curve. The algebraic curve itself may be too complicated to work with directly, but under some conditions, certain properties of the algebraic curve can still be seen in its tropical counterpart. This paper is concerned with when the property of `genus' is retained by tropicalization, and how to calculate it explicitly. Figure \ref{fig:elliptic_curve_example} illustrates an example of a tropical elliptic curve alongside a classical elliptic curve; in this case the genus of the elliptic curve is reflected as the unique hexagonal cycle in the tropical elliptic curve.

\begin{figure}
  \centering
  \begin{minipage}[b]{0.45\textwidth}
    \centering
  \begin{tikzpicture}
    \begin{axis}[
      xlabel={$x$},
      ylabel={$y$},
      xmin=-2, xmax=2,
      ymin=-2, ymax=2,
      axis equal image,
      axis lines=center,
      xtick=\empty,
      ytick=\empty,
      ]
      \draw[blue, ultra thick] (axis cs:-1.325,0.185) -- (axis cs:-1.325,-0.185);
      \addplot[domain=-2:2, samples=1000, blue, ultra thick] {(1 - x + x^3)^(1/3)};
      \addplot[domain=-2:2, samples=1000, blue, ultra thick] {-(1 -x + x^3)^(1/3)};
    \end{axis}
  \end{tikzpicture}
  \caption{A elliptic curve $C$}
  \end{minipage}
  \begin{minipage}[b]{0.45\textwidth}
    \centering
\begin{tikzpicture}
     \draw[blue, ultra thick] (-2,-2) -- (0,0);
     \draw[blue, ultra thick] (0,0) -- (0,1);
     \draw[blue, ultra thick] (0,1) -- (-1,1);
     \draw[blue, ultra thick] (-1,1) -- (-1,3);
     \draw[blue, ultra thick] (-1,1) -- (-2,0);
     \draw[blue, ultra thick] (0,1) -- (1,2);
     \draw[blue, ultra thick] (1,2) -- (1,4);
     \draw[blue, ultra thick] (1,2) -- (2,2);
     \draw[blue, ultra thick] (2,2) -- (3,3);
     \draw[blue, ultra thick] (3,3) -- (3,4);
     \draw[blue, ultra thick] (3,3) -- (4,3);
     \draw[blue, ultra thick] (2,2) -- (2,1);
     \draw[blue, ultra thick] (2,1) -- (4,1);
     \draw[blue, ultra thick] (2,1) -- (1,0);
     \draw[blue, ultra thick] (1,0) -- (0,0);
     \draw[blue, ultra thick] (1,0) -- (1,-1);
     \draw[blue, ultra thick] (1,-1) -- (0,-2);
     \draw[blue, ultra thick] (1,-1) -- (3,-1);
 \end{tikzpicture}
    \caption{Tropical elliptic curve $\tr(C)$}
  \end{minipage}
  \caption{Examples of classical and tropical elliptic curves}
  \label{fig:elliptic_curve_example}
\end{figure}

For an algebraic curve $\cC \subseteq \mathbb C\{\{t\}\}^n$ defined over the field of Puiseux series $\mathbb C\{\{t\}\}$, its tropicalization $\text{Trop}(\cC)$ is defined by removing any points where any coordinate is zero, then taking the valuations of each point coordinate-wise, and then taking the Euclidean closure, that is
$$\text{Trop}(\cC) := \overline{\{\text{Val}(p) : p \in \cC \cap (\mathbb C\{\{t\}\}^*)^n\}}$$
(as a set). This set is a union of bounded line segments we call edges and unbounded line segments we call rays. Now we assign a weight, which must be positive integer, to any edge or ray, as follows: take a point on the edge, and compute the initial ideal of the ideal of the curve with respect to the order function that assigns every coordinate function to the tropical coordinate of the point we have taken. Then the weight is defined as the sum of all multiplicities of all associated primes of the initial ideal (see \cite[Definition~3.4.3]{sturmfelsmaclagan}). In this paper, it will never be necessary to compute weights, since we will prove that the weights of all edges of the tropical curve that we consider are all equal to one.

We will need a notion of smoothness for a tropical curve - notions of smoothness exists for plane tropical curves (see for instance \cite{sturmfelsmaclagan}), but we define a notion of smoothness for tropical curves in any dimension, which is standard and equivalent to the other notions of smoothness for tropical curves. This definition makes use of tropical (iso)-morphisms, see for instance \cite[Definition 7.1.4]{mikhalkinrau}. A tropical morphism $f$ between two tropical varieties $T_1$ and $T_2$ is a continuous map which is locally affine $\mathbb Z$ - linear, that is, locally is given by composition of a $\mathbb Z$ - linear map with an arbitrary real translation. A tropical morphism $f:T_1 \rightarrow T_2$ is an isomorphism if there is a tropical morphism $g:T_2 \rightarrow T_1$ which is inverse to $f$, and such that the weights of each point in $T_1$ and $T_2$ agree under the maps.

\begin{definition}[Smooth Tropical Curve]\label{def:smooth_mikhalkin_rau}
    Let $T \subseteq \mathbb R^d$ be a tropical curve. We say that $T$ is \textit{smooth} if for each point $p \in T$ there exists a neighbourhood $U_p$ containing $p$, such that $U_p \cap T$ is tropically isomorphic to the intersection of a tropical line with an open set.
\end{definition}

The above definition can be thought of as `\textit{a tropical curve is smooth if it is locally isomorphic to a tropical line'}. To see why such a definition is important, let us see an example of an algebraic curve which is smooth, but whose tropicalization is \textit{not} smooth, and whose genus is not preserved under tropicalization.

We define the {\em tropical genus} of a connected tropical curve $T$ as the number of its (bounded) edges minus the number of its vertices plus one. 

\begin{remark}
    It is well-known that all weights of a smooth tropical curve are equal to one (see \cite[Definition~8.1.1]{mikhalkinrau}).
\end{remark}

\begin{example}\label{eg:Eg1}
    Consider the algebraic curve $V(y^3 + x^3 + 1)$, which is a smooth curve of genus one. The corresponding tropical curve is defined by the minimum $\min\{3y,3x,0\}$ being attained twice - this is precisely the same (as a set of points) as the usual tropical line in $\mathbb R^2$. It is, however, not smooth, since each edge has weight three. The underlying graph is a single vertex, and so has Betti number zero. So, the genus and the tropical genus are not equal.
\end{example}

\begin{figure}[H]
  \centering
  \begin{minipage}[b]{0.45\textwidth}
    \centering
  \begin{tikzpicture}
    \begin{axis}[
      xlabel={$x$},
      ylabel={$y$},
      xmin=-2, xmax=2,
      ymin=-2, ymax=2,
      axis equal image,
      axis lines=center,
      xtick=\empty,
      ytick=\empty,
      ]
      \draw[blue, ultra thick] (axis cs:-1,0.17) -- (axis cs:-1,-0.15);
      \addplot[domain=-2:2, samples=2000, blue, ultra thick] {(-1 - x^3)^(1/3)};
      \addplot[domain=-2:2, samples=1000, blue, ultra thick] {-(1 + x^3)^(1/3)};
    \end{axis}
  \end{tikzpicture}
  \caption{The curve $x^3 + y^3 + 1 = 0$}
  \end{minipage}
  \hfill
  \begin{minipage}[b]{0.45\textwidth}
    \centering
\begin{tikzpicture}
    \begin{axis}[
      xlabel={$x$},
      ylabel={$y$},
      xmin=-2, xmax=2,
      ymin=-2, ymax=2,
      axis equal image,
      axis lines=center,
      axis on top=true,
      xtick=\empty,
      ytick=\empty,
      ]
  \draw[blue, ultra thick] (axis cs:-2,-2) -- (axis cs:0,0);
  \draw[blue, ultra thick] (axis cs:2,0) -- (axis cs:0,0);
  \draw[blue, ultra thick] (axis cs:0,2) -- (axis cs:0,0);
  \node[above right] at (axis cs:0,1) {3};   
  \node[above right] at (axis cs:1,0) {3};
  \node[above right] at (axis cs:-1.2,-1) {3};
  \end{axis}
  \end{tikzpicture}
    \caption{The corresponding tropical curve}
  \end{minipage}
\end{figure}

\begin{theorem} \label{thm:smooth}
    Let $\cC$ be a smooth projective curve of genus $g$, such that $T:=\text{Trop}(\cC)$ is smooth. Then the tropical genus of $T$ is equal to $g$.
\end{theorem}

\begin{proof}
Essentially, this is a result in \cite{jell2020constructing}.
The proof uses several concepts from non-archimedean analytic geometry we did not define because we will not use them anywhere else (fully faithful tropicalization, analytification, Berkovich skeleta, Mumford curves). We refer to \cite{jell2020constructing,berkovich1990} for the definitions. 

By \cite[Theorem~5.7]{jell2020constructing}, smoothness of $T$ implies that the tropicalization is fully faithful. In particular, the tropical curve $T$ is homeomorphic to the Berkovich skeleton of the analytification of $\cC$ with respect to the given tropical embedding. 
By \cite[Theorem~5.6]{jell2020constructing}, the smoothness of the tropicalization also implies that $\cC$ is a Mumford curve, so all its Berkovich skeleta have genus $g$. Therefore $T$ has also genus $g$ as a graph, and the claim follows from Euler's formula.
\end{proof}

\section{The Tropical Generic Fiber} \label{sec:genus}

In this section we explain the tropicalization of the generic fiber of $f_G$, and prove that the tropical curve we obtain is smooth. (Note that this statement is stronger than smoothness of the generic fiber itself.)

\subsection{Tropicalization}

In its standard form, the generic fiber of $f_G$ is defined by distance equations of the form
$$(x_i-x_j)^2+(y_i-y_j)^2 = \lambda_e,$$
where $e=\{i,j\}$ runs over all edges except $\{1,2\}$, and $\lambda_e\in{\mathbb C}$ is a generic parameter.
It will be convenient for us to make a linear change of coordinates (see for instance \cite{GRASEGGER2020111713}), which will simplify all such equations. We define this change of variables in the following way. For all pairs $(i,j)$ such that $e:= \{i,j\} \in E\setminus\{\{1,2\}\}$, and $i < j$, we define
\begin{gather*}
    s_e := (x_j - x_i) + i(y_j - y_i) \\
    r_e := (x_j -x_i) - i(y_j - y_i) 
\end{gather*}
This gives us $2|E|-2$ variables. After making this change, we notice that for any cycle $\{i_1,i_2,...,i_k,i_{k+1} = i_1\}$ with edges $e_1=\{i_1,i_2\},\dots,e_k=\{i_k,i_1\}$ in $G$, we have that
$$\sum_{l=1}^k \pm s_{e_l} = \sum_{l=1}^k \pm r_{e_l} = 0.$$
The sign in the $l$-th summand is $+1$ if $i_l<i_{l+1}$ and $-1$ if $i_l>i_{l+1}$. The edge $e_0=\{1,2\}$ may occur in the cycle, and in this case we add a constant of $\pm 1$ to the sum. After making this change, the generic fiber of the edge map (i.e. the configuration space) is defined by the equations
$$s_{e_0}=r_{e_0} = 1, \quad s_{e}r_{e} = \lambda_{e}, \quad  \sum_{e \in C}s_{e} = \sum_{e \in C}r_{e} = 0,$$
where $C$ ranges over every cycle in the graph. Let us denote this fiber by $\cC$. It is isomorphic (via projection to the $s$-coordinates) to the intersection of two manifolds in ${\mathbb C}^n$, where $n:=|E|-1$. The first manifold, which we call $L_1$, is the intersection of the linear subspace defined by all equations of the form $\sum_{e \in C} s_{e} = 0$ with $\mathbb C^n$, and the second, denoted by $L_2$, is the image of $L_1$ under the map $({\mathbb C}^\ast)^n\to({\mathbb C}^\ast)^n$ given by $$(s_1,\dots,s_n)\mapsto (\lambda_1/s_1,\dots,\lambda_n/s_n).$$
In order to tropicalize the generic fiber, we first change the field from ${\mathbb C}$ to $\mathbb C \{ \{ t \} \}$, the field of complex Puiseux series. This does not change the number of irreducible components or the genus of $\cC$, as the two fields are actually isomorphic. We recall that the tropicalization of an algebraic variety  $V\subset (\mathbb C \{ \{ t \} \}^\ast)^n$ is defined as the Euclidean closure of the set of all value vectors of points in $V$. The value of a nonzero element in $\mathbb C \{ \{ t \} \}$ is its order with respect to $t$, which is simply a rational number. 

We tropicalize the defining equations for the generic fiber $\cC$ as follows. We label the (non-normalised) edges from $1$ to $n=|E|-1$; single subscripts now denote edges, and the normed edge is given index $0$. For each parameter $\lambda_e$, we just remember its value $w_e$ -- the numbers $w_1,\dots,w_n$ will be the tropical parameters of the tropical equations. Instead of the algebraic variables $s_1,\dots,s_n,t_1,\dots,t_n$ we take tropical variables $u_1,\dots,u_n,v_1,\dots,v_n$ ranging over ${\mathbb R}$ - the orientation of edges we had previously disappears at this point, since the valuation of a Puiseux series is not changed by scalar multiplication. The normed parameter $w_0$ corresponding to the edge $\{1,2\}$ and the norming variables $u_0,v_0$ are set to zero - the valuation of their algebraic value. Multiplication is replaced by addition, and a sum being zero is replaced by the condition that the minimum in some sequence of numbers is attained twice. We obtain the following system.
\begin{gather*}
    u_0 = v_0 = 0 \\
    \forall e \in E \setminus\{\{1,2\}\}, \quad u_e + v_e = w_e \\
    \text{For all cycles }C \text{ in } G, \ \ \min_{e \in C}\{u_e\} \text{ and } \min_{e \in C}\{v_e\} \text{ are attained at least twice}.
\end{gather*}
Furthermore, once the weights $w_e$ have been fixed, we may eliminate all of the $v_e$ variables via $v_e = w_e - u_e$, further reducing the ambient space to $\mathbb R^{n}$, given by the now fully reduced conditions 
\begin{gather*}
    u_0 = 0 \\
    \text{For all cycles }C \text{ in } G, \ \ \min_{e \in C}\{u_e\} \text{ and } \min_{e \in C}\{w_e -u_e\} \text{ are attained at least twice.}
\end{gather*}
The tropicalization of the generic fiber $\cC$ is contained in the zero set of the tropical equations above. We will see that it is actually equal to this zero set. To do this, we will use the following lemma.

\begin{lemma}\label{lem:bergman}
The tropicalization of the linear variety $L_1$ is equal to the tropical variety $X$ defined by the tropical equations
\begin{gather*}
    u_0 = 0 \\
    \text{For all cycles }C \text{ in } G, \ \ \min_{e \in C}\{u_e\} \text{ is attained twice.}
\end{gather*}
\end{lemma}

\begin{proof}
Because $L_1$ is a linear space, its tropicalization is defined by all circuits, i.e. minimal supports of linear equations in the ideal of $L_1$. By \cite[Example~4.2.14]{sturmfelsmaclagan}, the circuits are exactly the cycle conditions. 
\end{proof}

In Lemma~\ref{lem:bergman}, we implicitly defined the tropical variety $X$. Now we define $Y$ as the tropicalization of $L_2$. Clearly, this is just the image of $X$ under the map from ${\mathbb R}^n$ to itself that sends $(x_1,...,x_n)$ to $(w_1-x_1,...,w_n-x_n)$. This map is an example of a tropical isomorphism, to be defined in the next subsection; we will need tropical isomorphisms in order to show that the tropicalisation of $\cC$ is tropically smooth.

A vector $w\in{\mathbb R}^n$ is called {\em tropically generic} if its coordinates are linearly independent over ${\mathbb Q}$. If $w$ is tropically generic, then it cannot be a vector of values, but the definitions of the tropical varieties $X$ and $Y$ still remain valid.

If $X,Y\subset{\mathbb R}^n$ are tropical varieties, and $p\in X\cap Y$, then we say that $X$ and $Y$ intersect transversally at $p$ if $p$ is contained in a face $\sigma$ of $X$ and in a face $\tau$ of $Y$ such that the affine span of $\sigma\cup\tau$ is equal to ${\mathbb R}^n$. We say that $X$ and $Y$ intersect transverally if $X$ and $Y$ intersect transverally at every intersection point. The fact that in our case $X$ and $Y$ intersect transversally is very important to our proof, and is given by the following result.

\begin{lemma} \label{thm:transversal}
If $w$ is tropically generic, then $X$ and $Y$ intersect transversally at any of its intersection points.
\end{lemma}

\begin{proof}
 Let $p\in X\cap Y$. Let $\sigma$ be a face of $X$ and let $\tau$ be a face of $Y$, both containing $p$.
  Let $L_\sigma$ be the affine span of $\sigma$ and let $L_\tau$ be the affine span of $\tau$. 
  Then $L_\sigma$ contains zero, because zero is in the closure of every face of $X$. The translate
  $L_\tau-w$ also contains zero, by the same reason. So, both $L_\sigma$ and $L_\tau-w$ are
  vector spaces. Moreover, both vector spaces are defined over ${\mathbb Q}$. Let $W:=L_\sigma+(L_\tau-w)$.
  Since $p\in L_\sigma$ and $-p\in L_\tau$, it follows that $w\in W$. But $w$ is generic, and
  $W$ is defined over ${\mathbb Q}$. It follows that $W={\mathbb R}^n$, which means $X$ and
  $Y$ intersect transversally.   
\end{proof}

As a consequence of transversality, it follows from \cite[3.4.12]{sturmfelsmaclagan} that the tropicalization of the generic fiber is defined by the constraints we gave before Lemma~\ref{lem:bergman}. 

\begin{corollary} \label{cor:trans}
The tropicalization of the generic fiber $\cC$ is equal to the intersection of $X$ and $Y$ (as a set).
\end{corollary}

Tropical varieties are defined as {\em weighted} polyhedral complexes, i.e., the maximal faces may come with higher multiplicity. Corollary~\ref{cor:trans} does not exclude multiple edges, but we will see in the next subsection that they will not occur.

\subsection{Tropical Smoothness}

For our discussion, we know by Sard's theorem that the generic fiber is smooth and with the help of the following lemma we also show that the corresponding tropical curve is smooth as well.

\begin{lemma}\label{lem:tropicalcurveissmooth}
    Let $G$ be a 1-dof graph, and let $\cC$ be its generic fiber. Then the tropicalization $\text{Trop}(\cC)$ is a smooth tropical curve.
\end{lemma}

\begin{proof}
Recall that we may consider $\text{Trop}(\cC)$ as being defined by the conditions
\begin{gather*}
u_0 = 0 \\
\text{For all cycles }C \text{ in } G, \ \ \min_{e \in C}\{u_e\} \text{ and } \min_{e \in C}\{w_e -u_e\} \text{ are attained twice.}
\end{gather*}
This can further be considered as the intersection of the two tropical varieties $X$ and $Y$ defined by
\begin{gather*}
X = \left\{ \underline{u} \in \mathbb R^{|E|-1} : \text{for all cycles }C \text{ in } G, \ \ \min_{e \in C}\{u_e\} \text{ is attained twice}\right\}\\
Y = \left\{ \underline{u} \in \mathbb R^{|E|-1} : \text{for all cycles }C \text{ in } G, \ \ \min_{e \in C}\{w_e -u_e\} \text{ is attained twice}\right\}.
\end{gather*}
A main observation is that $Y$ is the image of $X$ after reflection in each coordinate axis ($\underline{u} \rightarrow -\underline{u}$) and translation by the weight vector $\underline{w} = (w_1,...,w_n)$. In order to prove that $V$ is smooth, we will use the following proposition \cite[Theorem 3.4.12]{sturmfelsmaclagan}.
\begin{proposition}\label{prop:transverseintersection}
    Let $W_1$ and $W_2$ be subvarieties of $(\mathbb C\{\{t\}\}^*)^n$. If the two tropical varieties $\text{Trop}(W_1)$ and $\text{Trop}(W_2)$ intersect transversally, then we have
    $$\text{Trop}(W_1) \cap \text{Trop}(W_2) = \text{Trop}(W_1 \cap W_2).$$
\end{proposition}
Therefore, as long as two tropical varieties intersect transversally, their intersection is equal to the tropicalization of the intersection of the two algebraic varieties. Recall that in our case, the tropical variety $X$ is the tropicalization of the linear variety $L_1$ defined as
$$L_1 := \left\{\underline{s} \in (\mathbb C\{\{t\}\}^*)^n : \text{for all cycles }C \text{ in } G, \ \ \sum_{e \in C}s_e = 0 \right\}.$$
The idea of our proof will be to show that for each point $p \in \text{Trop}(\cC)$, locally around $p$, $\text{Trop}(\cC)$ can be described as the transverse intersection of the tropicalization of two linear varieties which intersect in a line; therefore by Proposition \ref{prop:transverseintersection}, $\text{Trop}(\cC)$ is a tropical line locally around $p$, and is therefore smooth.

The two tropical varieties $X$ and $Y$ can also be considered as polyhedral complexes; by the genericity of the fiber, we may assume that no two codimension one faces of these complexes intersect. Therefore, a point $p \in X \cap Y$ is in one of the following three classes:
\begin{enumerate}
    \item $p$ lies in the intersection of a full dimensional face of $X$ and a full dimensional face of $Y$.
    \item $p$ lies in the intersection of a codimension one face of $X$ and a full dimensional face of $Y$.
    \item $p$ lies in the intersection of a full dimensional face of $X$ and a codimension one face of $Y$.
\end{enumerate}
We now make a further reduction; we note that via the map $\underline{u} \mapsto \underline{w} - \underline{u}$, which is a tropical isomorphism, we may interchange $X$ and $Y$. This allows us to reduce to only Cases 1 and 2, by implicitly applying this isomorphism to any Case 3 occurrence, changing it to Case 2.

In both cases, we take some small neighbourhood $U_p$ of $p$. We will reflect, pointwise, all of $Y$ through the point $p$ - let us call this map $\sigma$. Note that this is a tropical isomorphism, and in particular $U_p \cap Y = U_p \cap \sigma(Y)$. Therefore the intersection $U_p \cap X \cap Y$ is exactly $U_p \cap X \cap \sigma(Y)$. Furthermore, we know that $\sigma(Y)$ is simply a translation of $X$, and so both are tropicalizations of the linear spaces $X = \text{Trop}(L_1)$, $\sigma(Y) = \text{Trop}(L)$, for some linear space $L$. Since this intersection is transverse by Prop \ref{prop:transverseintersection}, we have that $X \cap \sigma(Y) = \text{Trop}(L_1 \cap L)$, which is the tropicalization of a line, as needed. This concludes the proof of Lemma \ref{lem:tropicalcurveissmooth}.
\end{proof}

As both the algebraic curve $\cC$ and its tropical counterpart $\text{Trop}(\cC)$ are smooth, we obtain the following corollary as a consequence of Theorem~\ref{thm:smooth}. We will abuse notation and let $\cC$ now refer to any single irreducible component of the generic fiber.

\begin{corollary}
Let $\cC$ be an irreducible component of the generic fiber with respect to a 1-dof graph $G$, and let $T$ be its tropicalization. Then the genus of $\cC$ is equal to the tropical genus of $T$.  \end{corollary}

Now that we have proven the theoretical foundation for our result, we give an algorithm which explicitly calculates $g(\text{Trop}(\cC))$.

\section{Algorithm}

In this section we describe an algorithm for computing $g(\tr(\cC))$. As mentioned previously, the algorithm will calculate the number of bounded edges and vertices of the tropical curve $\tr(\cC)$ by traversing the curve iteratively. The input of the algorithm are the edges of the graph $G$, and the output is the value $|E|-|V|+1$, where $E$ and $V$ refer to the bounded edges and vertices of $\tr(\cC)$. 

The first step of the algorithm is to find a good `starting point', meaning a point on $\tr(\cC)$, such that this particular curve exhibits the properties we need - namely, it is a transverse intersection of the corresponding $X$ and $Y$ from the preceding section. In particular, this entails finding a point $p \in \mathbb R^{|E|-1}$ such that for a `good' choice of weight vector $w$, the point $p$ lies on the tropical curve $\tr(\cC)$; specifically, this means that $p:=(x_1,...,x_d)$ (here indices correspond to edges of $G$) satisfies the cycle conditions of the form $\min_{e \in C}\{ x_e\}$ obtained twice, and that for some weight vector $w = (w_1,...,w_e)$ the cycle conditions $\min_{e \in C}\{ w_e - x_e\}$ obtained twice are also satisfied, and that furthermore this weight vector $w$ gives a transverse intersection of the complexes $X$ and $Y$ as defined in Section \ref{sec:genus}.

After the starting point is found, the main algorithm begins. The basic structure is shown in Algorithm \ref{alg1}. Within Algorithm \ref{alg1}, there are two sub-routines, Directions$(v)$ and Travel$(v,d)$. The travel subroutine is simple, and moves in direction $d$ from $v$ until it reaches a vertex. The more complex subroutine is Directions$(v)$, which calculates the directions of the edges leaving $v$, together with the distance required to travel in each direction to find the next vertex. This is the content of Lemma \ref{lem:dir}, and the method used by the algorithm can be seen in Algorithm \ref{algdirection}.
\begin{algorithm}
\caption{Traversing a Tropical Curve} \label{alg1} 
\begin{algorithmic}
\State Unexplored = \{Starting Point\} 
\State Vertices = \{Starting Point\}
\State Edges = \{ \}
\While{$\text{Unexplored} \neq \emptyset$}
    \For{$v \in \text{Unexplored}$}
        \State $D = \text{Directions}(v)$
        \For{$d \in D$}
            \If{$(v,d)$ gives infinite ray}
                \State Continue
            \EndIf
            \State $v' = \text{Travel}(v,d)$
            \If{$v' \notin \text{Vertices}$}
                \State $\text{Vertices} = \text{Vertices} \cup \{v'\}$
                \State $\text{Edges} = \text{Edges} \cup \{v,v'\}$
                \State $\text{Unexplored} = \text{Unexplored} \cup \{v'\}$
            \EndIf
        \EndFor
        \State $\text{Unexplored} = \text{Unexplored} \setminus \{v\}$
    \EndFor
\EndWhile
\State \Return{$|\text{Edges}| - |\text{Vertices}| +1$}
\end{algorithmic}    
\end{algorithm}
\subsection{Starting Point / Weight Generation}

Recall that Lemma~\ref{thm:transversal} guarantees that if we take a generic weight vector, then any point in the tropical fiber of $w$ determined by the tropical constraints given before Lemma~\ref{lem:bergman} is a good starting point. But solving these contraints may be costly. In this subsection we explain how to avoid that costly step.

\begin{lemma} \label{lem:start}
Let $u\in{\mathbb R}^n$ be a generic weight vector, and $X$ be the tropical variety as defined in Lemma \ref{lem:tropicalcurveissmooth}. Let $p\in{\mathbb R}^n$ be the unique intersection point of
$X$, the translate $u+X$, and the hyperplane $H$ defined by $x_n-1=0$. If $w:=2p-u$, then $w$ is a generic weight vector, and $p$ is a point in the intersection
of $X$ and $Y:=w-X$.
\end{lemma}

\begin{proof}
Since $p\in X\cap (u+X)$ and $Y=2p-(u+X)$, it follows that $p\in X\cap Y$, and the last claim is proved.
Note that $X$ and $u+X$ intersect transversally at $p$, because $u$ is generic. Because $\dim(X)+\dim(u+X)+\dim(H)=n$, the point $p$
lies in a maximal faces of $X$, $u+X$, and $H$ (the hyperplane $H$ has only a single face anyway).
So, there exists a neighborhood $U$ of $p\in{\mathbb R}^n$ such that $U\cap (u+X)$ is affine linear. Without loss of generality,
the neighborhood of $U$ may be assumed to be symmetric around $p$, i.e., $U=2p-U$. Then the tropical isomorphism
$t:z\to 2p-z$ maps $u+X$ to $Y$. This implies that $X$ and $Y$ intersect transversally at $p$.

Let $E\supseteq {\mathbb Q}$ be the field generated by the coordinates of $w$. Then both $X$ and $Y$ are defined over $E$, and therefore
the coordinates of $p$ all lie in $E$. Then also the coordinates of $u=2p-w$ also lie in $E$. But $u$ was chosen generically,
hence $\dim_{\mathbb Q}(E)=n$. This shows the remaining claim that $w$ is generic.
\end{proof}

Our algorithm is probabilistic, with a single random step: the choice of a random vector $u$ that plays the role of the generic translation
from $X$ to $u+X$. The conclusion in Lemma~\ref{lem:start} is fulfilled with $U$ in some open subset $W\subset{\mathbb R}^n$ that contains
all generic vectors. Its complement $F$ is again a tropical variety of dimension less than $n$ (see~\cite[Example~5.5.5]{sturmfelsmaclagan}).
Our probabilistic algorithm relies on the fact that a rational random vector will likely not lie in $F$.

Typically, the tropicalization of an algebraic system of equations is simpler than the variety itself. Linear systems are an exception to this rule of thumb. The tropical system for the starting point $p$ is the tropicalization of the following linear system defined over ${\mathbb C}\{t\}$:
\[ \sum_{e\in C} s_i= \sum_{e\in C} t_i = 0 \mbox{ for all cycles $C$,}
 \ \ s_it_i = t^{u_i}, \ s_0=t_0=1, \ s_n=t^1. \]
In order to compute $p$, we solve this linear system and set $p$ as the value vector of the unique solution. Since this requires Puiseux series arithmetic, this algorithm is implemented in Maple instead of Python.

\subsection{Direction Calculation}\label{sec:dir_calc}

A main step of our algorithm is to calculate the directions in which a tropical curve can continue from a vertex; in order to do this efficiently, we prove the following lemma restricting our search space to only the vectors with entries from $\{0,1\}$.

\begin{lemma}\label{lem:dir}
  Let $G$ be a 1-dof graph, and let $B$ be an edge in the tropical curve $\text{Trop}(\cC)$, where the weights $w_e$ are integers. Then the direction vector of $B$ is, up to scalar multiplication, a two-valued vector with entries in $\{0,1\}$.
\end{lemma}
\begin{proof}
    Take any edge $B$ of $\text{Trop}(\cC)\subseteq \mathbb R^d$, and take a point $x = (x_1,...,x_d)$ in the interior of this edge (note that the indices here also correspond to some edge $e \neq \{1,2\}$. On this edge, some coordinate entries may be constant; we assume that for those entries which are \textit{not} constant, that they take a non-integer value, and are not equal to any value of the constant entries - this can be achieved by moving a small amount on the edge. We then define
    $$\lceil x \rceil := \left(\lceil x_1 \rceil,...,\lceil x_d \rceil \right), \quad \lfloor x \rfloor := \left(\lfloor x_1 \rfloor,...,\lfloor x_d \rfloor \right).$$
 Both of these points lie on $\text{Trop}(\cC)$, since whichever value the minimum in any cycle condition $\min_{e \in C}\{ x_e \}$ takes, after taking floor functions the two (or more) edges where the minimum is attained, still attain the minimum. The same is true of the other cycle condition $\min_{e \in C} \{w_e - \lfloor x_e \rfloor \}$, with the same argument applied to $\lfloor w_e - x_e \rfloor = w_e - \lfloor x_e \rfloor$, which holds since $w_e$ is an integer. 

 We now claim that for sufficiently small $\lambda_0$, the line segment $$x + \lambda (\lceil x \rceil - \lfloor x \rfloor), \ \lambda \in (-\lambda_0,\lambda_0)$$
 is contained in $\text{Trop}(\cC)$. Indeed, we firstly note that at $x$, each cycle condition minimum is either being attained at an integer value, in which case since this line segment does not change the value of integer coordinates, the minimum is still being attained for $\lambda$ small enough. Secondly, if the minimum in some cycle condition at $x$ is being attained at a non-integer value, then all entries of $x$ with that non-integer value are being changed by value $\lambda (1,1,...,1)$, and thus the minimum is preserved for small enough $\lambda$ (we must ensure that none of these non-integer values are changed enough to become integers, since at these points they may become equal to the non-changing integer valued coordinates). Since the direction vector $\lceil x \rceil - \lfloor x \rfloor$ has entries in $\{0,1\}$, we are done.
 \end{proof}

 \begin{remark}
     We remark that when we explicitly calculate the direction vector of an edge, it may of course need to be scaled by some non-zero constant to have all non-zero entries equal to one. Therefore when we find the directions from a vertex, we must attempt to travel in both directions; we have lost the orientation of the direction by the rescaling implicitly involved.
 \end{remark}

 We can further restrict our search space by checking, for a vertex $v$, in which positions each cycle minimum is obtained twice. This is useful since it allows us to `connect together' coordinates, since some possible directions will break the minimum obtained twice condition. More precisely, suppose that at a vertex $v = (v_1,...,v_d)$, some cycle condition has its minimum being attained exactly twice, say at the coordinates $(v_s,v_t)$. Then for any edge coming out of $v$, if the coordinate for $v_s$ is changing along that edge, the coordinate for $v_t$ must be changing by the same amount - otherwise one would overtake the other, and the minimum would not be attained twice. We use the simple example of $C_4$ to illustrate this.

 For the graph $C_4$, with weights $w_1 = 1, w_2=2, w_3 = 3$, the two cycle conditions are
 \begin{gather*}
      \min\{x,y,z,0\} \text{ attained twice}\\
 \min\{1-x,2-y,3-z,0\} \text{ attained twice}
 \end{gather*}
 Suppose that at some stage of the algorithm we are at the vertex $(0,2,3)$ of the tropical curve in $\mathbb R^3$. From the first cycle condition, we see that the value of $x$ cannot change; this would break the minimum condition. From the second cycle condition, we see that if either of $y$ or $z$ were to decrease, then the other one must remain fixed. Finally, we see that if either of $y$ or $z$ were to increase, then so must the other, by the same amount. This means that only three directions are possible; they are $(0,1,1)$, $(0,-1,0)$, and $(0,0,-1)$. These are precisely the three direction vectors of the edges coming out of this vertex. Below is an outline of the algorithm, written in pseudocode.

 \begin{algorithm}
     \caption{Calculation of possible directions from a vertex $v$}\label{algdirection}
     \begin{algorithmic}
         \State Input: Vertex $v$ and the set \textbf{Equations} of cycle equations \textbf{Equations}
         \State \textbf{Pairs} = \{ \}
         \For{$eq \in$ \textbf{Equations}}
            \If{$v$ attains $eq$ min exactly twice, at entries $v_i,v_i$}
                \State $\textbf{Pairs} = \textbf{Pairs} \cup \{\{v_i,v_j\}\}$
            \EndIf
         \EndFor
         \State \textbf{Blocks} $=$ Transitive closure of \textbf{Pairs}
         \State \textbf{Directions} = {0,1 vectors with constant entries along each block in \textbf{Blocks}}
     \end{algorithmic}
 \end{algorithm}
 
\section{The Genus of a Trace}\label{sec:4 couplercurves}

In \cite{lubbes2024open}, the authors ask for the genus of the trace of a vertex of a 1-dof graph when an edge of the graph is
fixed in position. The trace of a vertex is simply the projection of the configuration set to the coordinates of that vertex. The trace might be a finite set of points; this happens in the case when the fixed edge and the tracing point are contained
in a rigid subgraph. 
If the trace has dimension one, then its irreducible components are curves that all have the same genus. Any irreducible
component of the configuration set projects to an irreducible component of the trace. Since the genus never increases
along an algebraic map between two algebraic curves, it follows that the genus of an irreducible component of the configuration set
is an upper bound. Also, the upper bound is an equality if the projection map, restricted to one of the components,
is birational to the image.

Throughout this section, we assume that $G$ is a 1-dof graph, $i,j,k$ are three distinct vertices, the fixed edge is $\{i,j\}$,
and the tracing vertex is $k$. The graph $G'$ obtained by adding edges $\{i,k\}$ and $\{j,k\}$ is called the {\em extended graph} $G'$.
Since the edge lengths of $\{i,k\}$ and $\{j,k\}$ may be considered as local coordinates of the tracing vertex, the tracing set
is finite if and only if both edge lengths have only finite values in the configuration set. This is the case if and only if
the extended graph $G'$ is not rigid. From now on, we assume that the tracing set has dimension one, or equivalently that
the extended graph $G'$ is rigid.

Following \cite{jacksonjordan}, we say that a graph is {\em globally rigid} if, for a generic realisation, all other realisations
with the same edge lengths can be obtained by rotation or reflection.

\begin{proposition} \label{prop:ggr}
Assume that the extended graph $G'$ is globally rigid. Then
the projection map from the configuration set to the tracing set is generically injective.
In particular, the restriction to any irreducible component is birational to its image.
\end{proposition}

\begin{proof}
Let $K$ be the configuration set for fixed generic edge lengths of the graph $G$ and fixed vertices $i,j$. Let $p_k$ be a generic point
of the tracing set. Together with the lengths of $\{i,k\}$ and $\{j,k\}$, the fixed edge lengths form an edge length vector of
$G'$ that is generic in the image of the edge map. If $G'$ is globally rigid, there are exactly two realisations of $G'$ with
this edge length vector and fixed vertices $i,j$. Moreover, one realisation is a reflection of the other. So, only one of them
maps the vertex $k$ to the position $p_k$. This is then the only point in the preimage of $p_k$ under the projection map.
\end{proof}

Figure~\ref{fig:coupler} shows the trace of a vertex of a graph with 6 vertices such that the extended graph is globally rigid. The genus of the configuration sets and the traces can be computed using Algorithm~\ref{alg1}.
\begin{figure}[H]
\begin{centering}
\includegraphics[height=4cm]{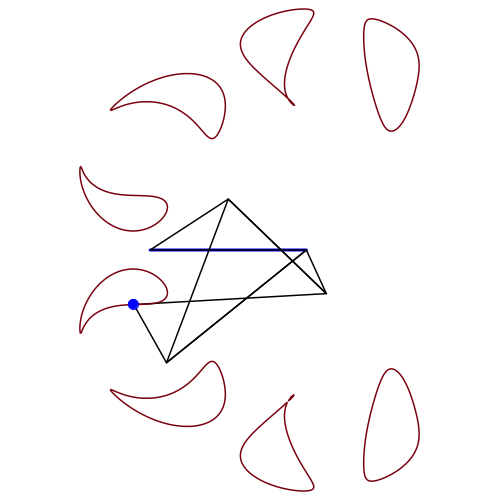} \hspace{1cm}
\includegraphics[height=4cm]{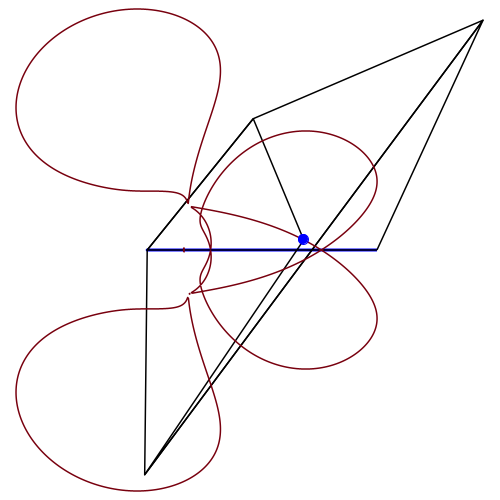} \hspace{1cm}
\includegraphics[height=4cm]{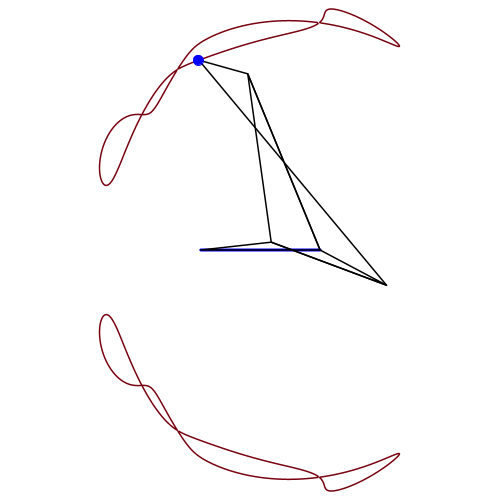}
\caption{Three traces of the same graph with six vertices. The genus of the configuration curves and the traces is equal to 17.}
\label{fig:coupler}
\end{centering}
\end{figure}
If the extended graph $G'$ is not globally rigid, then it is sometimes possible to replace the graph $G$ by a graph
with fewer vertices that produces the same tracing set for $k$. For instance, if $G$ is a 4-cycle, then the tracing set is a circle,
and this circle can also be obtained by a subgraph with three vertices. We call such a process a simplification.

\begin{definition}[Simplifications]\label{def:simp} We define two types of simplifications.\begin{enumerate}
\item Assume that $H$ is a (minimally) rigid subgraph of $G$ with two vertices $u,v\in V(H)$ such that any edge of $G$ connecting a vertex in
$V(G)\setminus V(H)$ with a vertex in $V(H)$, then the vertex in $V(H)$ is either $u$ or $v$; moreover, we assume that
that $V(H)\setminus\{u,v\}$ does not contain $i$, $j$, or $k$. Then the graph $G_1$ (which will be a 1-dof graph) obtained from $G$ by replacing the subgraph $H$
by a single edge $\{u,v\}$ is called a {\em simplification of the first kind}.
\item Assume that $G_2$ is a proper 1-dof subgraph of $G$ such that $i,j,k\in V(G_1)$. Then $G_2$ is called a
{\em simplification of the second kind}.
\end{enumerate}

\end{definition}

If $G_1$ is a simplification of the first kind, then the tracing set of $G$ is the union of the tracing sets of $G_1$ for different
edge lenths of the new edge $\{u,v\}$.
If $G_2$ is a simplification of the second kind, then the tracing set of $G$ is a subset of the tracing set of $G_1$.
In both cases, any component of the tracing set of $G$ is a component of a tracing set of the simplification.

It is known that global rigidity is preserved under \emph{1-extensions} and \emph{edge additions} \cite{jacksonjordan}.

A graph is called {\em 3-connected} iff it has at least four vertices, and any subgraph obtained by removing two vertices
is connected. A graph is called {\em redundantly rigid} if any subgraph obtained by removing an edge is rigid.
By \cite[Theorem~7.1]{jacksonjordan}, a graph with at least four vertices is globally rigid if and only if
it is 3-connected and redundantly rigid.
We also recall the theorem of Geiringer/Laman (\cite{Geiringer,Laman}): a graph $H$ such that $|E(H)|=2|V(H)|-3$ is minimally
rigid if and only if $|E(H')|\le 2|V(H')|-3$ for every subgraph $H'$ of $H$ with at least two vertices.

\begin{lemma}\label{lem:hyp}
Let $H$ by a rigid graph such that $|E(H)|=2|V(H)|-2$ that is not redundantly rigid. Then $H$ has a proper rigid subgraph $H_1$
such that $|E(H_1)|=2|V(H_1)|-2$.
\end{lemma}

\begin{proof}
Let $e\in E(H)$ such that the graph $H_0$ obtained from $H$ by removing $e$ is not rigid. Then $|E(H_0)|=2|V(H_0)|-3$,
and by the theorem of Geiringer/Laman there exists a subgraph $H_1$ of $H_0$ such that $|E(H_1)|>2|V(H_1)|-3$.
Let $e'\in E(H)$ be an edge such that the graph $H_2$ obtained from $H$ by removing $e'$ is minimally rigid -- such an
edge exists because $H$ is rigid but not minimally rigid. Let $H_3$ be the graph obtained from $H_1$ by removing $e'$ in
case $e'\in E(H_1)$, otherwise we set $H_3=H_1$. Then $H_3$ is a subgraph of $H_2$, and $|E(H_3)|\le 2|V(H_3)|-3$
by Geiringer/Laman. Therefore $e'\in E(H_1)$ and $|E(H_1)|=2|V(H_1)|-2$. Also, the graph $H_3$ is minimally rigid
by Geiringer/Laman, and therefore $H_1$ is rigid as well.
\end{proof}

\begin{proposition} \label{prop:simpl}
If the graph $G'$ is not globally rigid, then there exists a simplification of the first or second kind.
\end{proposition}

\begin{proof}
By the theorem of Jackson/Jord\`an, the graph $G'$ is not 3-connected or not redundantly rigid. We distinguish these two cases.

Case~1: $G'$ is not 3-connected. Then there exist $u,v\in V(G)$ such that the full subgraph of $G'$ with vertex set $V(G)-\{u,v\}$
is not connected. From the connected components and the vertices $u$ and $v$, we can form two proper subgraphs $G_1,G_2$ of $G'$
such that $E(G')$ is the disjoint union of $E(G_1)$ and $E(G_2)$, and $V(G_1)\cap V(G_2)=\{u,v\}$. At least one of $G_1,G_2$ is
rigid, because otherwise $G'$ could not be rigid, contrary to one of the assumptions in the beginning of this section.
Without loss of generality, we assume that $G_2$ is rigid. Since the 3-cycle with vertices $i,j,k$ is contained in $G'$, all
its three vertices must either be all in $G_1$ or all in $G_2$. They cannot be in $G_2$ because then the tracing set would
be finite. Hence they are all in $G_1$. And then $G_1$ minus the two edges $\{i,k\}$ and $\{j,k\}$ is a simplification of the first kind.

Case~2: $G'$ is not redundantly rigid. If either $\{i,k\}$ or $\{j,k\}$ is already in $G$, then the trace is a circle,
and the full subgraph with vertices $i,j,k$ is a simplification of the second kind. Similarily, if $G'$ minus the edge $\{i,k\}$
is not rigid, then the trace is a circle, and we get a simplification as well. So we assume $\{i,k\}$ and $\{j,k\}$ are
not in $G$ and that the two graphs $G'$ minus $\{i,k\}$ and $G'$ minus $\{j,k\}$ are rigid; then the two graphs are necessarily
minimally rigid. The first statement implies that $|E(G')=2|V(G')|-2$. By Lemma~\ref{lem:hyp}, $G'$ has a rigid subgraph $G_1$
such that $|E(G_1)|=2|V(G_1)|-2$. By Geiringer/Laman, $G_1$ cannot be a subgraph of ($G'$ minus $\{i,k\}$). Hence $\{i,k\}\in E(G_1)$.
Similarily, $\{j,k\}\in E(G_1)$. So, $G_1$ minus the two edges $\{i,k\}$ and $\{j,k\}$ is a simplification of the second kind.
\end{proof}
Proposition~\ref{prop:ggr} and Proposition~\ref{prop:simpl} together imply that the genus of a component of a trace
is equal to the genus of a component of the configuration curve, after fully simplifying the graph (meaning 
repeatedly finding simplifications until the extended graph is globally rigid). Hence we can use
Algorithm~\ref{alg1} for computing the genus of a trace.

\section{Examples}

In this section we give a fully worked example for $C_4$, and then give the genus of the generic configuration space for various other graphs. For the family of graphs $K_{2,n}$ we prove a formula for the genus using the Riemann-Hurwitz formula, and show that it agrees with our computations.

\subsection{The graph \texorpdfstring{$C_4$}{C4}}

In the case of the four-cycle $C_4$, we normalise the edge variables $s_0 = t_0 = 1$, and after picking generic positive valuation Puiseux series $\lambda_1,\lambda_2,\lambda_3$, the generic fiber for the edge map $f_{C_4}$ is the set of $(s_1,s_2,s_3,t_1,t_2,t_3) \in (\mathbb C\{\{t\}\}^*)^6$ that satisfy the conditions
\begin{gather*}
    s_1t_1 = \lambda_1, \quad s_2t_2 = \lambda_2, \quad s_3t_3 = \lambda_3 \\
    1 + s_1 + s_2 + s_3 = 0, \quad 1 + t_1 + t_2 + t_3 = 0.
\end{gather*}
After tropicalization, these equations become
\begin{gather*}
    u_0 = v_0 = 0, \quad u_1+v_1 = w_1, \quad u_2+v_2 = w_2, \quad u_3+v_3 = w_3 \\
    \min\{0,u_1,u_2,u_3\} \text{ occurs twice,} \quad  \min\{0,v_1,v_2,v_3\} \text{ occurs twice.} 
\end{gather*}
Finally, after using the first three equations to eliminate the $v$ variables, we find the two tropical varieties $X$ and $Y$ as in Section \ref{sec:genus}. We rename the three remaining variables to $x,y,z$ for convenience.
\begin{gather*}
    X:= \left\{(x,y,z) \in \mathbb R^3 : \min\{0,x,y,z\} \text{ occurs twice}\right\} \\
    Y:= \left\{(x,y,z) \in \mathbb R^3 : \min\{0,w_1 - x,w_2-y,w_3 -z\} \text{ occurs twice}\right\}.
\end{gather*}
The tropical curve we are interested in is the intersection of these two polyhedral complexes. They are shown below, together with their intersection. 
\begin{figure}[H]
  \centering
  \begin{minipage}[b]{0.49\textwidth}
    \centering
  \includegraphics[scale=0.4]{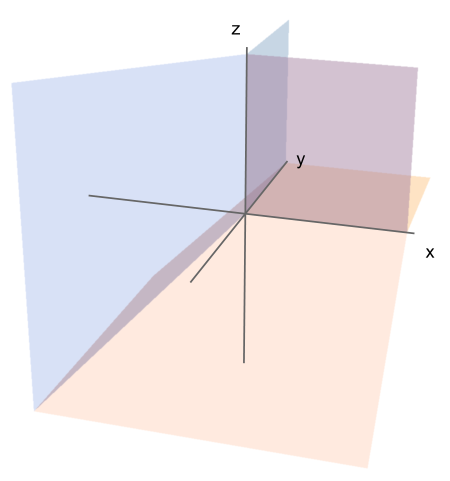}
    \caption{The Polyhedral Complex $X$}
  \end{minipage}
  \hfill
  \begin{minipage}[b]{0.49\textwidth}
    \centering
  \includegraphics[scale=0.4]{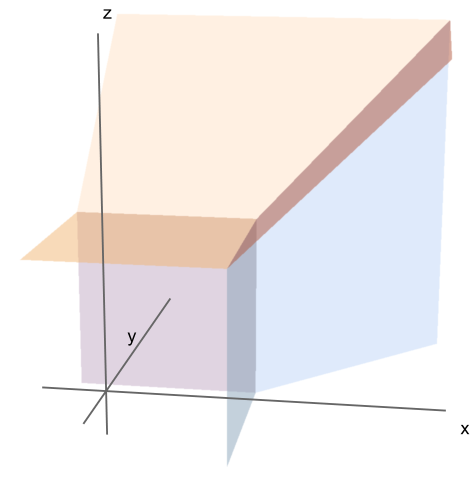}
    \caption{The Polyhedral Complex $Y$}
  \end{minipage}
\end{figure}
\begin{figure}[H]
    \centering
    \includegraphics[scale=0.4]{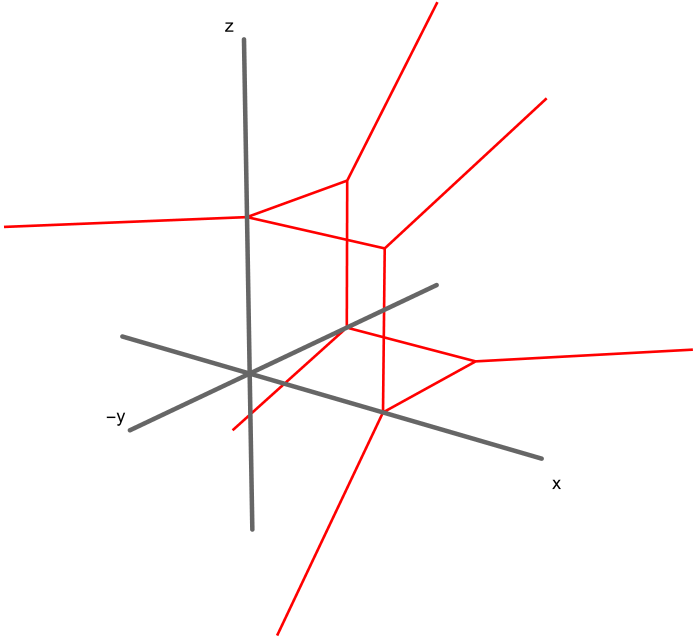}
    \caption{The tropical curve $X \cap Y$ corresponding to $C_4$}
    \label{fig:c4tropical}
\end{figure}
We see from Figure \ref{fig:c4tropical} that the underlying graph of $X \cap Y$ is $C_6$, which has Betti number one, confirming that the graph $C_4$ has generic fiber of genus one.

\subsection{The Bipartite Graphs \texorpdfstring{$K_{2,m}$}{K2m}}

Generalising the example $C_4 = K_{2,2}$ above, we can use our algorithm to calculate the genus of the generic fiber for the bipartite graphs $K_{2,m}$, for $m$ up to 11. The values are shown below.

\begin{table}[H]
    \centering 
    \begin{tabular}{c|ccccccccccc}
        $m$& 1 & 2 & 3 & 4 & 5 & 6 & 7 & 8 & 9 & 10 & 11\\
        \hline
        $g(K_{2,m})$ & 0 & 1 & 5 & 17 & 49 & 129 & 321 & 769 & 1793 & 4097 & 9217\\
    \end{tabular}
    \caption{The genus of the bipartite graphs $K_{2,m}$ calculated by the algorithm.}
    \label{K2mvalues}
\end{table}

These values can be verified, and further a closed form for the genus $g(K_{2,m})$ can be given by using the Riemann-Hurwitz formula. Let $C_{K_{2,m}}$ denote the curve given by the generic fiber of $K_{2,m}$, and we consider $K_{2,m}$ as being obtained by a single Henneberg 1-move (also called a 0-extension) from $K_{2,m-1}$, which creates the new vertex $w = (z_1,z_2)$. Applying Riemann-Hurwitz to the projection 
$$\pi: C_{K_{2,m}} \rightarrow C_{K_{2,m-1}}, (x_1,y_1,...,x_{m-1},y_{m-1},z_1,z_2) \mapsto (x_1,y_1,...,x_{m-1},y_{m-1})$$
we see that this is a branched covering of degree $2$, whose ramification points correspond to realisations of $K_{2,m}$ where the new vertex $w$ is collinear with its neighbours $v_1$ and $v_2$. Such realisations occur in two scenarios: 
$$d(v_1,v_2)^2 = (d(v_1,w) + d(w,v_2))^2, \quad d(v_1,v_2)^2 = (d(v_1,w) - d(w,v_2))^2$$
which gives two possible values for the squared distance between $v_1$ and $v_2$. The number of times one of these equations is satisfied in $K_{2,m}$ is the realisation number of the graph $K_{2,m} \cup \{v_1,v_2\}$, where the squared distance between $v_1$ and $v_2$ is chosen as the one being tested for. This graph is a collection of $m$ triangles with a common edge; however the triangle containing $w$ is degenerate (by the collinearity of $v_1$, $v_2$, and $w$), and therefore this graph has realisation number $2^{m-1}$. Doing this for each of the two possible squared distance, we find that the number of ramification points is $2^m$. The Riemann-Hurwitz formula then gives
$$2g(C_{K_{2,m}}) - 2 = 2 (2g(C_{K_{2,m-1}}) - 2) + 2^m.$$
Solving this recurrence gives the closed form 
$$g(C_{K_{2,m}}) = (m-2)2^{m-1} + 1$$
which agrees with the calculated values in Table \ref{K2mvalues}.

\begin{remark}
    The graphs $K_{2,m}$, $m\geq 2$ fit into a larger, two parameter family of graphs $O_{l,r}$, where the graph $O_{l,r}$ is obtained from the four cycle $K_{2,2}$ by performing $l$ 0-extensions on the two `left' vertices of the partition, and $r$ 0-extensions on the two `right' vertices of the partition. Note that it is irrelevant in which order these extensions are performed, and that $O_{l,r}$ is a bipartite graph with $4 + l + r$ vertices. The graphs $K_{2,m}$ are the subfamily of graphs $O_{m-2,0}$. 

    An almost identical application of Riemann-Hurwitz to the above, shows that the behaviour of the genus of these graphs is independent of the side on which the 0-extension is performed; that is, we have
    $$g(C_{O_{l,r}}) = g(C_{O_{l-1,r+1}})= ...= g(C_{K_{2,l+r+2}}).$$
\end{remark}

\subsection{Further Examples}

Using our algorithm, we have calculated the genus for various graphs. Below is a list of graphs together with their genus. Note that we do not list 1-dof graphs which are simply a smaller 1-dof graph with a rigid part attached along an edge, such as the $C_4$ with a triangle on one of its edges. A rigid graph $R$ attached to an edge of a 1-dof graph does not change the genus of (a simple component of) the generic fiber - this can be seen by applying Riemann-Hurwitz to the projection which removes $R$ - this map has degree $1$ with no branching points, when applied to a single component of the generic configuration space of the larger graph.

\vspace{5mm}

\begin{minipage}{0.3\textwidth}
\centering
    \begin{tikzpicture}[scale = 0.7]
        \draw[gray, thick] (-1,0) -- (1,0);
        \draw[gray, thick] (-1,0) -- (-2,-1);
        \draw[gray, thick] (-2,1) -- (-2,-1);
        \draw[gray, thick] (1,0) -- (2,-1);
        \draw[gray, thick] (2,1) -- (2,-1);
        \draw[gray, thick] (-2,1) -- (2,1);
        \draw[gray, thick] (-2,-1) -- (2,-1);
        \draw[gray, thick] (2,1) -- (1,0);
        \filldraw[black] (-1,0) circle (2pt);
        \filldraw[black] (1,0) circle (2pt);
        \filldraw[black] (-2,-1) circle (2pt);
        \filldraw[black] (-2,1) circle (2pt);
        \filldraw[black] (2,-1) circle (2pt);
        \filldraw[black] (2,1) circle (2pt);
    \end{tikzpicture}
    
    Genus 5
\end{minipage}
\begin{minipage}{0.3\textwidth}
\centering
    \begin{tikzpicture}[scale = 0.7]
        \draw[gray, thick] (-1,0) -- (1,0);
        \draw[gray, thick] (-1,0) -- (-2,-1);
        \draw[gray, thick] (-2,1) -- (-1,0);
        \draw[gray, thick] (1,0) -- (2,-1);
        \draw[gray, thick] (2,1) -- (2,-1);
        \draw[gray, thick] (-2,1) -- (2,1);
        \draw[gray, thick] (-2,-1) -- (2,-1);
        \draw[gray, thick] (2,1) -- (1,0);
        \filldraw[black] (-1,0) circle (2pt);
        \filldraw[black] (1,0) circle (2pt);
        \filldraw[black] (-2,-1) circle (2pt);
        \filldraw[black] (-2,1) circle (2pt);
        \filldraw[black] (2,-1) circle (2pt);
        \filldraw[black] (2,1) circle (2pt);
    \end{tikzpicture}
    
    Genus 5
\end{minipage}
\begin{minipage}{0.3\textwidth}
    \centering
    \begin{tikzpicture}[scale = 0.7]
        \draw[gray, thick] (-1.5,0) -- (-0.5,1);
        \draw[gray, thick] (-0.5,1) -- (0.5,1);
        \draw[gray, thick] (0.5,1) -- (1.5,0);
        \draw[gray, thick] (-1.5,0) -- (-0.5,-1);
        \draw[gray, thick] (-0.5,-1) -- (0.5,-1);
        \draw[gray, thick] (1.5,0) -- (0.5,-1);
        \draw[gray, thick] (0.5,1) -- (1.5,0);
        \draw[gray, thick] (0.5,1) -- (0,0);
        \draw[gray, thick] (-0.5,1) -- (0,0);
        \draw[gray, thick] (0.5,-1) -- (0,0);
        \draw[gray, thick] (-0.5,-1) -- (0,0);
        \filldraw[black] (-1.5,0) circle (2pt);
        \filldraw[black] (-0.5,1) circle (2pt);
        \filldraw[black] (0.5,1) circle (2pt);
        \filldraw[black] (1.5,0) circle (2pt);
        \filldraw[black] (-0.5,-1) circle (2pt);
        \filldraw[black] (0.5,-1) circle (2pt);
        \filldraw[black] (0,0) circle (2pt);
    \end{tikzpicture}
    
    \centering
    Genus 5
\end{minipage}
\vspace{5mm}

\begin{minipage}{0.3\textwidth}
    \centering
    \begin{tikzpicture}[scale = 0.7]
        \draw[gray, thick] (-1,0) -- (1,0);
        \draw[gray, thick] (-1,0) -- (-2,-1);
        \draw[gray, thick] (-2,1) -- (-2,-1);
        \draw[gray, thick] (1,0) -- (2,-1);
        \draw[gray, thick] (2,1) -- (2,-1);
        \draw[gray, thick] (-2,1) -- (-1,0);
        \draw[gray, thick] (-2,-1) -- (2,-1);
        \draw[gray, thick] (2,1) -- (1,0);
        \draw[gray, thick] (2,1) -- (0,1.5);
        \draw[gray, thick] (-2,1) -- (0,1.5);
        \filldraw[black] (-1,0) circle (2pt);
        \filldraw[black] (1,0) circle (2pt);
        \filldraw[black] (-2,-1) circle (2pt);
        \filldraw[black] (-2,1) circle (2pt);
        \filldraw[black] (2,-1) circle (2pt);
        \filldraw[black] (2,1) circle (2pt);
        \filldraw[black] (0,1.5) circle (2pt);
    \end{tikzpicture}
    
    \centering
    Genus 7
\end{minipage}
\begin{minipage}{0.3\textwidth}
\centering
    \begin{tikzpicture}[scale = 0.7]
        \draw[gray, thick] (-1,0) -- (1,0);
        \draw[gray, thick] (-1,0) -- (0,1);
        \draw[gray, thick] (1,0) -- (0,1);
        \draw[gray, thick] (1,0) -- (1,-1);
        \draw[gray, thick] (-1,0) -- (-1,-1);
        \draw[gray, thick] (-1,-1) -- (1,-1);
        \draw[gray, thick] (-2,0) -- (0,1);
        \draw[gray, thick] (2,0) -- (0,1);
        \draw[gray, thick] (2,0) -- (1,-1);
        \draw[gray, thick] (-2,0) -- (-1,-1);
        \filldraw[black] (-1,0) circle (2pt);
        \filldraw[black] (-1,-1) circle (2pt);
        \filldraw[black] (1,0) circle (2pt);
        \filldraw[black] (0,1) circle (2pt);
        \filldraw[black] (1,-1) circle (2pt);
        \filldraw[black] (-2,0) circle (2pt);
        \filldraw[black] (2,0) circle (2pt);
    \end{tikzpicture}
    
    Genus 17
\end{minipage}
\begin{minipage}{0.3\textwidth}
\centering
    \begin{tikzpicture}[scale = 0.7]
        \draw[gray, thick] (-1.5,0) -- (-0.5,1);
        \draw[gray, thick] (-0.5,1) -- (0.5,1);
        \draw[gray, thick] (0.5,1) -- (1.5,0);
        \draw[gray, thick] (-1.5,0) -- (0,-1);
        \draw[gray, thick] (1.5,0) -- (0,-1);
        \draw[gray, thick] (0.5,1) -- (1.5,0);
        \draw[gray, thick] (0,-1) -- (0,0);
        \draw[gray, thick] (-1.5,0) -- (0,0.5);
        \draw[gray, thick] (1.5,0) -- (0,0.5);
        \draw[gray, thick] (0,0.5) -- (0,0);
        \draw[gray, thick] (0,0) -- (-0.5,1);
        \filldraw[black] (-1.5,0) circle (2pt);
        \filldraw[black] (-0.5,1) circle (2pt);
        \filldraw[black] (0.5,1) circle (2pt);
        \filldraw[black] (1.5,0) circle (2pt);
        \filldraw[black] (0,-1) circle (2pt);
        \filldraw[black] (0,0) circle (2pt);
        \filldraw[black] (0,0.5) circle (2pt);
    \end{tikzpicture}
    
    Genus 49
\end{minipage}

\vspace{5mm}

\begin{minipage}{0.3\textwidth}
\centering
    \begin{tikzpicture}[scale = 0.7]
        \draw (0,0) to[out=80,in=180] (2,2);
        \draw[gray, thick] (0,0) -- (1,1);
        \draw[gray, thick] (0,0) -- (3,0);
        \draw[gray, thick] (0,0) -- (0,3);
        \draw[gray, thick] (1,1) -- (2,1);
        \draw[gray, thick] (1,1) -- (1,2);
        \draw[gray, thick] (1,2) -- (2,2);
        \draw[gray, thick] (2,1) -- (2,2);
        \draw[gray, thick] (3,0) -- (3,3);
        \draw[gray, thick] (2,2) -- (3,3);
        \draw[gray, thick] (2,1) -- (3,0);
        \draw[gray, thick] (1,2) -- (0,3);
        \filldraw[black] (0,0) circle (2pt);
        \filldraw[black] (3,0) circle (2pt);
        \filldraw[black] (0,3) circle (2pt);
        \filldraw[black] (3,3) circle (2pt);
        \filldraw[black] (1,1) circle (2pt);
        \filldraw[black] (2,1) circle (2pt);
        \filldraw[black] (1,2) circle (2pt);
        \filldraw[black] (2,2) circle (2pt);
    \end{tikzpicture}
    
    Genus 129
\end{minipage}
\begin{minipage}{0.3\textwidth}
\centering
    \begin{tikzpicture}[scale = 0.7]
        \draw[gray, thick] (2,4) -- (3,4);
        \draw[gray, thick] (3,4) -- (4,3);
        \draw[gray, thick] (4,2) -- (4,3);
        \draw[gray, thick] (4,2) -- (3,1);
        \draw[gray, thick] (3,1) -- (2,1);
        \draw[gray, thick] (1,2) -- (2,1);
        \draw[gray, thick] (1,2) -- (1,3);
        \draw[gray, thick] (1,3) -- (2,4);
        \draw[gray, thick] (2,4) -- (3,1);
        \draw[gray, thick] (3,4) -- (2,1);
        \draw[gray, thick] (4,2) -- (1,3);
        \draw[gray, thick] (4,3) -- (1,2);
        \filldraw[black] (2,4) circle (2pt);
        \filldraw[black] (3,4) circle (2pt);
        \filldraw[black] (4,3) circle (2pt);
        \filldraw[black] (4,2) circle (2pt);
        \filldraw[black] (3,1) circle (2pt);
        \filldraw[black] (2,1) circle (2pt);
        \filldraw[black] (1,3) circle (2pt);
        \filldraw[black] (1,2) circle (2pt);

    \end{tikzpicture}
    
    Genus 225
\end{minipage}
\begin{minipage}{0.3\textwidth}
\centering
    \begin{tikzpicture}[scale = 0.7]
        \draw[gray, thick] (0,0) -- (1,1);
        \draw[gray, thick] (0,0) -- (3,0);
        \draw[gray, thick] (0,0) -- (0,3);
        \draw[gray, thick] (1,1) -- (2,1);
        \draw[gray, thick] (1,1) -- (1,2);
        \draw[gray, thick] (1,2) -- (2,2);
        \draw[gray, thick] (2,1) -- (2,2);
        \draw[gray, thick] (3,0) -- (3,3);
        \draw[gray, thick] (0,3) -- (3,3);
        \draw[gray, thick] (2,2) -- (3,3);
        \draw[gray, thick] (2,1) -- (3,0);
        \draw[gray, thick] (1,2) -- (0,3);
        \filldraw[black] (0,0) circle (2pt);
        \filldraw[black] (3,0) circle (2pt);
        \filldraw[black] (0,3) circle (2pt);
        \filldraw[black] (3,3) circle (2pt);
        \filldraw[black] (1,1) circle (2pt);
        \filldraw[black] (2,1) circle (2pt);
        \filldraw[black] (1,2) circle (2pt);
        \filldraw[black] (2,2) circle (2pt);
        
    \end{tikzpicture}
    
    Genus 247
\end{minipage}

\vspace{5mm}

\begin{minipage}{0.3\textwidth}
\centering
    \begin{tikzpicture}[scale = 0.7]
        \draw[gray, thick] (-1,0) -- (1,0);
        \draw[gray, thick] (-1,0) -- (-2,-1);
        \draw[gray, thick] (-2,1) -- (-1,0);
        \draw[gray, thick] (1,0) -- (2,-1);
        \draw[gray, thick] (2,1) -- (2,-1);
        \draw[gray, thick] (-2,1) -- (2,1);
        \draw[gray, thick] (-2,-1) -- (2,-1);
        \draw[gray, thick] (2,1) -- (1,0);
        \draw[gray, thick] (-2,1) -- (-2,-1);
        \draw[gray, thick] (0,1) -- (0,-1);
        \filldraw[black] (-1,0) circle (2pt);
        \filldraw[black] (1,0) circle (2pt);
        \filldraw[black] (-2,-1) circle (2pt);
        \filldraw[black] (-2,1) circle (2pt);
        \filldraw[black] (2,-1) circle (2pt);
        \filldraw[black] (2,1) circle (2pt);
        \filldraw[black] (0,0) circle (2pt);
        \filldraw[black] (0,1) circle (2pt);
        \filldraw[black] (0,-1) circle (2pt);
    \end{tikzpicture}
    
    Genus 85
\end{minipage}
\begin{minipage}{0.3\textwidth}
\centering
    \begin{tikzpicture}[scale = 0.6]
        \draw[gray, thick] (0,3) -- (1,0);
        \draw[gray, thick] (0,3) -- (1.5,2.5);
        \draw[gray, thick] (0,3) -- (3,4);
        \draw[gray, thick] (1,0) -- (5,0);
        \draw[gray, thick] (1,0) -- (2,1);
        \draw[gray, thick] (5,0) -- (4,1);
        \draw[gray, thick] (5,0) -- (6,3);
        \draw[gray, thick] (6,3) -- (4.5,2.5);
        \draw[gray, thick] (6,3) -- (3,4);
        \draw[gray, thick] (1.5,2.5) -- (4.5,2.5);
        \draw[gray, thick] (1.5,2.5) -- (4,1);
        \draw[gray, thick] (3,3) -- (2,1);
        \draw[gray, thick] (3,3) -- (4,1);
        \draw[gray, thick] (4.5,2.5) -- (2,1);
        \draw[gray, thick] (3,4) -- (3,3);

        \draw[gray, thick] (2,1) -- (4,1);
        \filldraw[black] (0,3) circle (2pt);
        \filldraw[black] (1,0) circle (2pt);
        \filldraw[black] (1.5,2.5) circle (2pt);
        \filldraw[black] (3,3) circle (2pt);
        \filldraw[black] (3,4) circle (2pt);
        \filldraw[black] (4,1) circle (2pt);
        \filldraw[black] (5,0) circle (2pt);
        \filldraw[black] (4.5,2.5) circle (2pt);
        \filldraw[black] (6,3) circle (2pt);
        \filldraw[black] (2,1) circle (2pt);
        
    \end{tikzpicture}
    
    Genus 785
\end{minipage}
\begin{minipage}{0.3\textwidth}
\centering
    \begin{tikzpicture}[scale = 0.7]
        \draw[gray, thick] (0,0) -- (1,1);
        \draw[gray, thick] (0,0) -- (3,0);
        \draw[gray, thick] (0,0) -- (0,3);
        \draw[gray, thick] (1,1) -- (2,1);
        \draw[gray, thick] (1,1) -- (1,2);
        \draw[gray, thick] (1,2) -- (2,2);
        \draw[gray, thick] (2,1) -- (2,2);
        \draw[gray, thick] (3,0) -- (3,3);
        \draw[gray, thick] (2,2) -- (3,3);
        \draw[gray, thick] (2,1) -- (3,0);
        \draw[gray, thick] (1,2) -- (0,3);
        \draw[gray, thick] (1.5,2.5) -- (2,2);
        \draw[gray, thick] (1.5,2.5) -- (1,2);
        \draw[gray, thick] (1.5,3) -- (0,3);
        \draw[gray, thick] (1.5,3) -- (3,3);
        \draw[gray, thick] (1.5,2.5) -- (1.5,3);
        \filldraw[black] (0,0) circle (2pt);
        \filldraw[black] (3,0) circle (2pt);
        \filldraw[black] (0,3) circle (2pt);
        \filldraw[black] (3,3) circle (2pt);
        \filldraw[black] (1,1) circle (2pt);
        \filldraw[black] (2,1) circle (2pt);
        \filldraw[black] (1,2) circle (2pt);
        \filldraw[black] (2,2) circle (2pt);
        \filldraw[black] (1.5,2.5) circle (2pt);
        \filldraw[black] (1.5,3) circle (2pt);
        
    \end{tikzpicture}
    
    Genus 795
\end{minipage}

\begin{minipage}{0.3\textwidth}
\centering
    \begin{tikzpicture}[scale = 0.7]
        \draw[gray, thick] (0,0) -- (1,1);
        \draw[gray, thick] (0,0) -- (3,0);
        \draw[gray, thick] (0,0) -- (0,3);
        \draw[gray, thick] (1,1) -- (2,1);
        \draw[gray, thick] (1,1) -- (1,2);
        \draw[gray, thick] (1,2) -- (2,2);
        \draw[gray, thick] (2,1) -- (2,2);
        \draw[gray, thick] (3,0) -- (3,3);
        \draw[gray, thick] (2,2) -- (3,3);
        \draw[gray, thick] (2,1) -- (3,0);
        \draw[gray, thick] (1,2) -- (0,3);
        \draw[gray, thick] (1.5,2.5) -- (2,2);
        \draw[gray, thick] (1.5,2.5) -- (1,2);
        \draw[gray, thick] (1.5,3) -- (0,3);
        \draw[gray, thick] (1.5,3) -- (3,3);
        \draw[gray, thick] (1.5,2.5) -- (1.5,3);
        \draw[gray, thick] (1.5,0.5) -- (1.5,0);
        \draw[gray, thick] (1.5,0.5) -- (1,1);
        \draw[gray, thick] (1.5,0.5) -- (2,1);
        \filldraw[black] (0,0) circle (2pt);
        \filldraw[black] (3,0) circle (2pt);
        \filldraw[black] (0,3) circle (2pt);
        \filldraw[black] (3,3) circle (2pt);
        \filldraw[black] (1,1) circle (2pt);
        \filldraw[black] (2,1) circle (2pt);
        \filldraw[black] (1,2) circle (2pt);
        \filldraw[black] (2,2) circle (2pt);
        \filldraw[black] (1.5,2.5) circle (2pt);
        \filldraw[black] (1.5,3) circle (2pt);
        \filldraw[black] (1.5,0.5) circle (2pt);
        \filldraw[black] (1.5,0) circle (2pt);
    \end{tikzpicture}
    
    Genus 2391
\end{minipage}
\begin{minipage}{0.3\textwidth}
\centering
    \begin{tikzpicture}[scale = 0.7]
        \draw[gray, thick] (0,0) -- (1,1);
        \draw[gray, thick] (0,0) -- (3,0);
        \draw[gray, thick] (0,0) -- (0,3);
        \draw[gray, thick] (1,1) -- (2,1);
        \draw[gray, thick] (1,1) -- (1,2);
        \draw[gray, thick] (1,2) -- (2,2);
        \draw[gray, thick] (2,1) -- (2,2);
        \draw[gray, thick] (3,0) -- (3,3);
        \draw[gray, thick] (2,2) -- (3,3);
        \draw[gray, thick] (2,1) -- (3,0);
        \draw[gray, thick] (1,2) -- (0,3);
        \draw[gray, thick] (1.5,2.5) -- (2,2);
        \draw[gray, thick] (1.5,2.5) -- (1,2);
        \draw[gray, thick] (1.5,3) -- (0,3);
        \draw[gray, thick] (1.5,3) -- (3,3);
        \draw[gray, thick] (1.5,2.5) -- (1.5,3);
        \draw[gray, thick] (2.5,1.5) -- (3,1.5);
        \draw[gray, thick] (2.5,1.5) -- (2,2);
        \draw[gray, thick] (2.5,1.5) -- (2,1);
        \filldraw[black] (0,0) circle (2pt);
        \filldraw[black] (3,0) circle (2pt);
        \filldraw[black] (0,3) circle (2pt);
        \filldraw[black] (3,3) circle (2pt);
        \filldraw[black] (1,1) circle (2pt);
        \filldraw[black] (2,1) circle (2pt);
        \filldraw[black] (1,2) circle (2pt);
        \filldraw[black] (2,2) circle (2pt);
        \filldraw[black] (1.5,2.5) circle (2pt);
        \filldraw[black] (1.5,3) circle (2pt);
        \filldraw[black] (2.5,1.5) circle (2pt);
        \filldraw[black] (3,1.5) circle (2pt);
    \end{tikzpicture}
    
    Genus 2277
\end{minipage}
\begin{minipage}{0.3\textwidth}
\centering
    \begin{tikzpicture}[scale = 0.5]
        \draw[gray, thick] (0,0) -- (1,1);
        \draw[gray, thick] (0,0) -- (3,0);
        \draw[gray, thick] (0,0) -- (0,3);
        \draw[gray, thick] (1,1) -- (2,1);
        \draw[gray, thick] (1,1) -- (1,2);
        \draw[gray, thick] (1,2) -- (2,2);
        \draw[gray, thick] (2,1) -- (2,2);
        \draw[gray, thick] (3,0) -- (3,3);
        \draw[gray, thick] (2,2) -- (3,3);
        \draw[gray, thick] (2,1) -- (3,0);
        \draw[gray, thick] (1,2) -- (0,3);
        \draw[gray, thick] (0,3) -- (3,3);
        \draw[gray, thick] (-1,-1) -- (4,-1);
        \draw[gray, thick] (-1,-1) -- (-1,4);
        \draw[gray, thick] (-1,4) -- (4,4);
        \draw[gray, thick] (4,-1) -- (4,4);
        \draw[gray, thick] (3,3) -- (4,4);
        \draw[gray, thick] (-1,-1) -- (0,0);
        \draw[gray, thick] (4,-1) -- (3,0);
        \draw[gray, thick] (-1,4) -- (0,3);
        \filldraw[black] (0,0) circle (2pt);
        \filldraw[black] (3,0) circle (2pt);
        \filldraw[black] (0,3) circle (2pt);
        \filldraw[black] (3,3) circle (2pt);
        \filldraw[black] (1,1) circle (2pt);
        \filldraw[black] (2,1) circle (2pt);
        \filldraw[black] (1,2) circle (2pt);
        \filldraw[black] (2,2) circle (2pt);
        \filldraw[black] (-1,-1) circle (2pt);
        \filldraw[black] (4,4) circle (2pt);
        \filldraw[black] (-1,4) circle (2pt);
        \filldraw[black] (4,-1) circle (2pt);
    \end{tikzpicture}
    
    Genus 13777
\end{minipage}
\vspace{5mm}

In addition to the genus, the code outputs an image of the tropical curve, with the infinite rays removed. Some example images are shown below.

\begin{minipage}{0.45\textwidth}
\centering
        \includegraphics[width=1\linewidth]{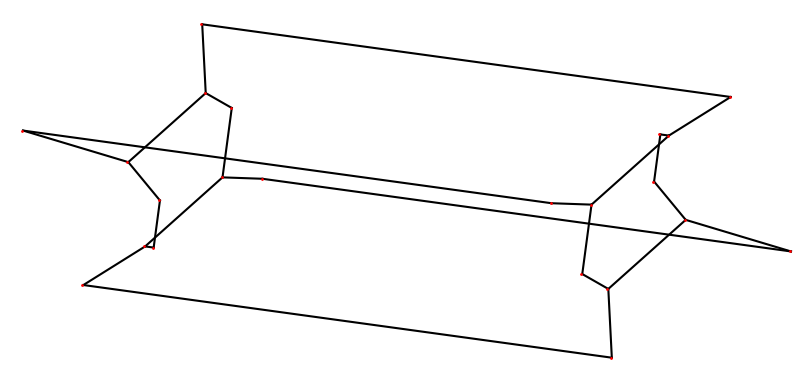}
        \captionof*{figure}{The tropical curve of $K_{2,3}$ \\
        Genus 5}
\end{minipage}
\begin{minipage}{0.45\textwidth}
\centering
        \includegraphics[width=1\linewidth]{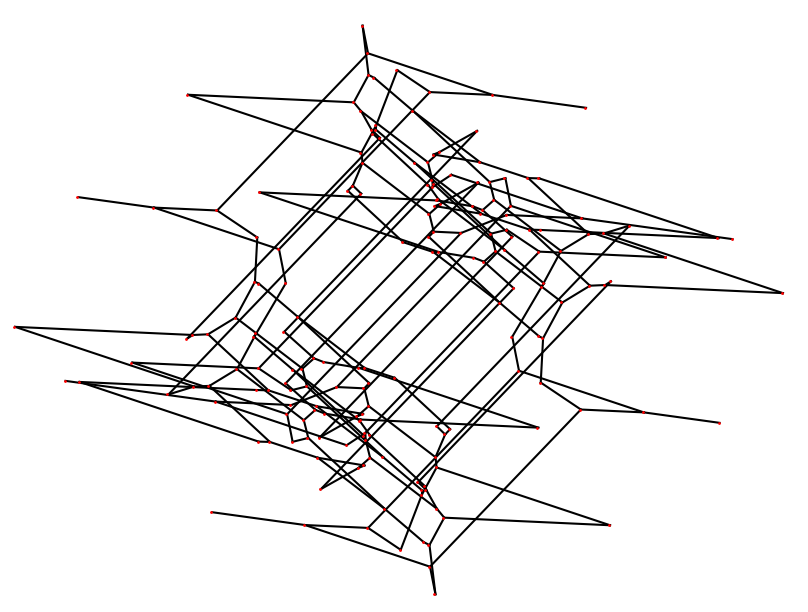}
        \captionof*{figure}{The tropical curve of the diamond graph \\
        (Row 2, column 3 above, genus 49)}
\end{minipage}

\vspace{5mm}

As a concluding remark, we observe that for all examples we computed above, the genus of the configuration space is odd. It seems reasonable to conjecture that for all 1-dof graphs, the configuration space for generic edge lengths has odd (or zero) genus. 

This does not contradict Kempe's universality theorem, which states that any algebraic curve in the plane is the trace of a vertex of some graph. If we would give a plane algebraic curve of even genus, our results in Section \ref{sec:4 couplercurves} would imply that it cannot be the trace of a generic 1-dof graph - otherwise after simplifying we would find a (generic) 1-dof graph with configuration set of even genus. This is not a contradition with the conjecture, rather it implies that such traces are given by non-generic configuration sets of graphs.

\begin{figure}[H]
\centering
        \includegraphics[width=0.7\linewidth]{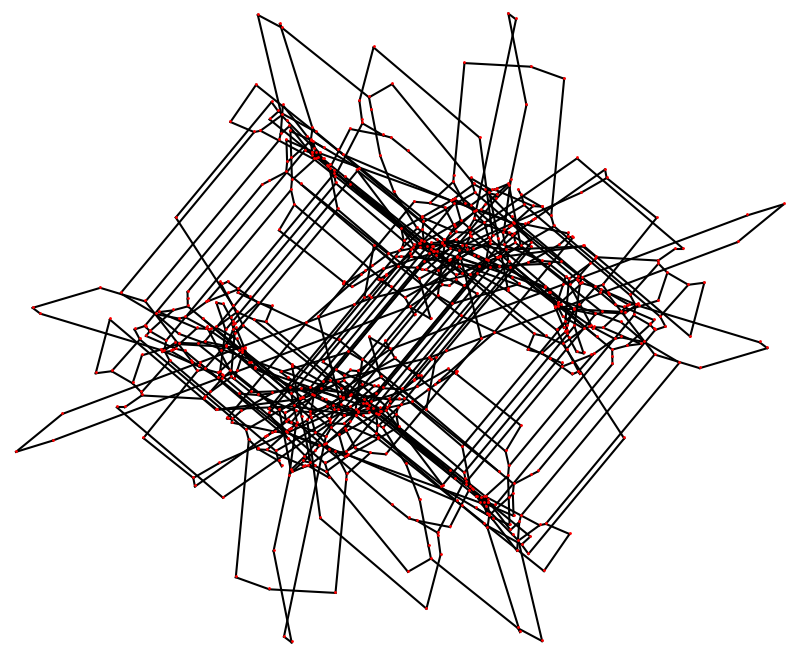}
        \caption{The tropical curve of the Wagner graph \\
        Genus 225}
\end{figure}

\section*{Acknowledgements}

The authors would like to thank Jan Legersk\'{y}, Matteo Gallet, and many guests of the RICAM Special Semester on Rigidity and Flexibility for helpful conversations.

\vspace{5mm}
\noindent Research Institute for Symbolic Computation (RISC)
\\Johannes Kepler University, Linz, Austria
\\\url{josef.schicho@risc.jku.at}
\\[5mm]
Johann Radon Institute for Computational and Applied Mathematics (RICAM)
\\Austrian Academy of Sciences, Linz, Austria
\\\url{ayushkumar.tewari@ricam.oeaw.ac.at}
\\\url{audie.warren@oeaw.ac.at}

\bibliography{calligraphgenus}

\begin{thebibliography}{10}

\bibitem{dewar2024tropical}
F.~Babaee, S.~Dewar, and J.~Maxwell.
\newblock Extremal decompositions of tropical varieties and relations with
  rigidity theory.
\newblock {\em arXiv Preprints: 2403.00655}, 2024.

\bibitem{berkovich1990}
V.~G. Berkovich.
\newblock {\em Spectral theory and analytic geometry over non-{A}rchimedean
  fields}, volume~33 of {\em Mathematical Surveys and Monographs}.
\newblock American Mathematical Society, Providence, RI, 1990.

\bibitem{bernstein2019tropical}
D.~Bernstein and R.~Krone.
\newblock The tropical {C}ayley-{M}enger variety.
\newblock {\em SIAM J. Discr. Math.}, 33:1725--1742, 2019.

\bibitem{borcea2004laman}
C.~Borcea and I.~Streinu.
\newblock The number of embeddings of minimally rigid graphs.
\newblock {\em Discrete {\&} Computational Geometry}, 31(2):287--303, Feb 2004.

\bibitem{capco2018number}
J.~Capco, M.~Gallet, G.~Grasegger, C.~Koutschan, N.~Lubbes, and J.~Schicho.
\newblock The number of realizations of a {L}aman graph.
\newblock {\em SIAM Journal on Applied Algebra and Geometry}, 2(1):94--125,
  2018.

\bibitem{lubbes2024open}
G.~Grasegger, B.~El~Hilany, and N.~Lubbes.
\newblock Coupler curves of moving graphs and counting realizations of rigid
  graphs.
\newblock {\em Math. Comp.}, 93:459--504, 2024.

\bibitem{GRASEGGER2020111713}
G.~Grasegger, J.~Legerský, and J.~Schicho.
\newblock Graphs with flexible labelings allowing injective realizations.
\newblock {\em Discrete Mathematics}, 343(6):111713, 2020.

\bibitem{izmestiev2020elliptic}
I.~Izmestiev.
\newblock Four-bar linkages, elliptic functions, and flexible polyhedra.
\newblock {\em Comp. Aided Geom. Design}, 79, 2020.

\bibitem{jacksonjordan}
B.~Jackson and T.~Jord{\'{a}}n.
\newblock Connected rigidity matroids and unique realisations of graphs.
\newblock {\em J. Comb. Theory B}, 94:1--24, 2005.

\bibitem{jackson2019laman}
B.~Jackson and J.~Owen.
\newblock Equivalent realisations of a rigid graph.
\newblock {\em Discrete Applied Mathematics}, 256:42--85, 2019.

\bibitem{jell2020constructing}
P.~Jell.
\newblock Constructing smooth and fully faithful tropicalizations for {Mumford}
  curves.
\newblock {\em Selecta Mathematica}, 26(4):60, 2020.

\bibitem{Laman}
G.~Laman.
\newblock On graphs and rigidity of plane skeletal structures.
\newblock {\em Journal of Engineering Mathematics}, 4:331--340, 1970.

\bibitem{lubbes2024irreducible}
N.~Lubbes, M.~Makhul, J.~Schicho, and A.~Warren.
\newblock Irreducible components of sets of points in the plane that satisfy
  distance conditions.
\newblock {\em arXiv Preprints: 2403.00392}, 2024.

\bibitem{mikhalkinrau}
G.~Mikhalkin and J.~Rau.
\newblock {\em Tropical Geometry}.
\newblock Textbook in Preparation.

\bibitem{Geiringer}
H.~{Pollaczek-Geiringer}.
\newblock {\"Uber die Gliederung ebener Fachwerke}.
\newblock {\em {Zeitschrift f\"ur Angewandte Mathematik und Mechanik (ZAMM)}},
  7:58--72, 1927.

\bibitem{sturmfelsmaclagan}
B.~Sturmfels and D.~Maclagan.
\newblock {\em Introduction to Tropical Geometry}.
\newblock American Mathematical Society, 2015.

\end{thebibliography}
\end{document}